\documentclass[12pt]{article}
\usepackage{amsmath,amsfonts,amssymb,amsthm,graphicx,float,color,caption,verbatim}
\usepackage{mathrsfs}\usepackage{hyperref}

\usepackage{geometry}
\usepackage{authblk}
\theoremstyle{plain}% default
\newtheorem{theorem}{Theorem}[section]
\newtheorem{lemma}[theorem]{Lemma}

\theoremstyle{definition}
\newtheorem{definition}[theorem]{Definition}

\theoremstyle{remark}

\title{Initial boundary value problems for a fractional differential equation with hyper-Bessel operator}
\author [*]{Fatma Al-Musalhi}
\author [*]{Nasser Al-Salti}
\author [*]{Erkinjon Karimov} 
\affil[*] {Department of Mathematics, Sultan Qaboos University, P.O. Box 36 Al-Khoudh, Oman}

\begin{document}
\maketitle
\begin{abstract}
Direct and  inverse source problems of a fractional diffusion equation with regularized Caputo-like counterpart hyper-Bessel operator are considered. Solutions to these problems are constructed based on appropriate eigenfunction expansion and results on existence and uniqueness are established. To solve the resultant equations, a solution to a non-homogeneous fractional differential equation with regularized Caputo-like counterpart hyper-Bessel operator is also presented.
\end{abstract}

\section{Introduction}
In this paper, we consider the following fractional differential equation involving hyper-Bessel operator with a source term $F$:
\begin{equation}
\,^{C}\left(t^{\theta}\dfrac{\partial}{\partial t}\right)^{\alpha}u(x,t)-u_{xx}(x,t)=F.
\end{equation}
We study both a direct problem where $F=f(x,t)$ is a known function of space and time and an inverse source problem where  $F=f(x)$ is an unknown function of space only.
Here $\,^{C}\left(t^{\theta}\dfrac{\partial}{\partial t}\right)^{\alpha}$ stands for a regularized Caputo-like counterpart hyper-Bessel operator of order $0<\alpha<1$ (see formula $(\ref{relation})$).
The hyper-Bessel operator was introduced by Dimovski in \cite{dim} and it arises in various problems, such as, fractional relaxation  \cite{garra} and fractional diffusion models \cite{garra1}. As an example, authors in \cite{garra1} used hyper-Bessel operator to describe heat diffusion for fractional Brownian motion. Their analysis based on converting fractional power of hyper-Bessel operator into Erd\'elyi-Kober (E-K) fractional integral. For more details about fractional Brownian motion, the reader is referred to  \cite{garra1, fbm}.\\
In fact, expressing hyper-Bessel operator in terms of Erd\'elyi-Kober fractional integral plays a key role in finding solution to fractional differential equations involving hyper-Bessel operator as illustrated in this paper as well. Some results related to hyper-Bessel operator are given in  \cite{garra},\cite{kiry2}.\\
In \cite{kiry}, AL-Saqabi and Kiryakova considered Volterra integral equation of second kind and a fractional differential equation, involving (E-K) integral or differential operator. They found explicit solutions to these equations using transmutation method which reduces solutions to known integral  solutions of Riemann-Liouville fractional equations.\\
The purpose of this paper is to prove existence and uniqueness of solution to a fractional diffusion equation involving a regularized Caputo-like counterpart hyper-Bessel operator considering both direct and inverse source problems.

\section{Preliminaries }
In this section, we recall some definitions and results related to fractional hyper-Bessel operator which will be used later in this paper. We start by writing down the definition of Erd\'elyi-Kober fractional integral.
\begin{definition}(see \cite{kiry, garra})
Erd\'elyi-Kober fractional integral of a function $f(t)$ with arbitrary parameters
$\delta>0$, $\gamma\in\mathrm{R}$ and $\beta >0$ is defined as
\begin{equation}
I_{\beta}^{\gamma,\delta}f(t)=\dfrac{t^{-\beta(\gamma+\delta)}}{\Gamma(\delta)}\int_{0}^{t}(t^{\beta}-\tau^{\beta})^{\delta-1} \tau^{\beta\gamma}f(\tau)\, d(\tau^{\beta}),
\end{equation}
which is reduced to the well-known Riemann-Liouville fractional integral when $\gamma= 0$ and $\beta =1$ with a power weight.\\
For $\delta<0,$ the interpretation is via integro-differential operator
$$I_{\beta}^{\gamma, \delta}f(t)=(\gamma+\delta+1)I_{\beta}^{\gamma,\delta+1}f(t)+\dfrac{1}{\beta}I_{\beta}^{\gamma,\delta+1} \left( t \dfrac{d}{dt}f(t)\right).$$
\end{definition}
In the following theorem, we present an explicit solution to an integral equation involving E-K fractional integral.
\begin{theorem}\label{eksol}(see \cite{kiry}, Theorem 1)
The unique solution $y(t)\in C_{\beta\mu},$ $\mu\geq \max\left\lbrace 0,-\gamma\right\rbrace-1$ to the following fractional  integral equation  of a second kind :
$$y(t)-\lambda t^{\beta \delta}I_{\beta}^{\gamma,\delta}y(t)=f(t),$$  or equivalently,
$$y(t)-\lambda t^{-\beta\gamma}\int_{0}^{t}\dfrac{(t^{\beta}-\tau^{\beta})^{\delta-1}}{\Gamma(\delta)}\tau^{\beta\gamma}y(\tau)\, d(\tau^{\beta})=f(t),$$
with $f\in C_{\beta\mu}$, has the explicit form of a convolutional type integral :
\begin{equation}y(t)=f(t)+\lambda t^{-\beta \gamma}\int_{0}^{t}(t^{\beta}-\tau^{\beta})^{\delta-1}E_{\delta,\delta}\left[ \lambda(t^{\beta}-\tau^{\beta})^{\delta}\right]\tau^{\beta\gamma}f(\tau)\, d(\tau^{\beta}). \end{equation}
\end{theorem}
Next, we use E-K integral to define the regularized counterpart hyper-Bessel operator.
\begin{definition}(see \cite{garra})
The hyper-Bessel operator of order $0<\alpha<1$ is defined in terms of E-K integral as follows
\begin{eqnarray}
\left(t^{\theta}\dfrac{d}{dt}\right)^{\alpha}f(t)=\left\lbrace \begin{array}{ll}
(1-\theta)^{\alpha}t^{-(1-\theta)\alpha} I_{1-\theta}^{0,-\alpha}f(t),& \text{ if } \theta<1,\\
(\theta-1)^{\alpha}I_{1-\theta}^{-1,-\alpha}t^{(1-\theta)\alpha}f(t),&\text{ if } \theta>1.
\end{array}\right.
\end{eqnarray}
\end{definition}
Note that when $\theta=0$, the hyper-Bessel operator coincides with Riemann-Liouville fractional derivative.\\
Now, recall that  Caputo and  Riemann-Liouville fractional derivatives of order $0<\alpha<1$ are defined as (see \cite{gm}): \[
\begin{array}{l}\displaystyle ^CD_{0|t}^{\alpha}f(t)=\dfrac{1}{\Gamma(1-\alpha)}\displaystyle\int_{0}^{t}\frac{f' (\tau)}{(t-\tau)^{\alpha}}\, d\tau,\\
 \displaystyle D_{0|t}^{\alpha}f(t)=\dfrac{1}{\Gamma(1-\alpha)}\displaystyle\dfrac{d}{dt}\int_{0}^{t}\frac{f (\tau)}{(t-\tau)^{\alpha}}\, d\tau,\end{array}\]
respectively, and they are related by (\cite{gm}):
$$
 \displaystyle ^CD_{0|t}^{\alpha}f(t)=D_{0|t}^{\alpha}(f(t)-f(0^{+})).
$$
Using the above relation, we can express the regularized Caputo-like counterpart hyper-Bessel operator as :
\begin{equation}\label{relation}
\,^{C}\left(t^{\theta}\dfrac{d}{dt}\right)^{\alpha}f(t)=\left(t^{\theta}\dfrac{d}{dt}\right)^{\alpha}f(t)-\dfrac{f(0)\,t^{-\alpha(1-\theta)}}{(1-\theta)^{-\alpha}\Gamma(1-\alpha)},
\end{equation}
and in terms of E-K fractional integral :
\begin{equation}\label{relation1}\,^{C}\left(t^{\theta}\dfrac{d}{dt}\right)^{\alpha}f(t)=(1-\theta)^{\alpha}t^{-\alpha(1-\theta)}I_{1-\theta}^{0,-\alpha} \left( f(t)-f(0) \right).\end{equation}
Also, we need to recall the Mittag-Leffler function of one parameter :
$$
E_{\alpha}(z)=\sum_{k=0}^{\infty}\dfrac{z^k}{\Gamma(\alpha k +1)},\quad\text{Re}(\alpha)>0,\,  z\in C,
$$
and the Mittag-Leffler of two parameters
$$
E_{\alpha,\beta^{*}}(z)=\sum_{k=0}^{\infty}\dfrac{z^k}{\Gamma(\alpha k +\beta^{*})},\quad\text{Re}(\alpha)>0,\text{ Re}(\beta^{*})>0,\,  z\in C.
$$
Now, we need the following results related to the  Mittag-Leffler function 
\begin{theorem}\label{intint}(see\cite{Prabhakar})
If $\text{Re}(\mu)>0$, $\text{Re}(\beta^*)>0,$ $\lambda$ is a complex number
and  $f(t)$ is an integrable function, then 
\begin{equation*}\label{simplified}
\begin{array}{ll}
\displaystyle\int_{a}^{x} (x-u)^{\beta^*-1}E_{\alpha,\beta^*}\left( \lambda(x-u)^{\alpha}\right) du& \displaystyle\int_{a}^{u} \dfrac{(u-t)^{\mu-1}}{\Gamma(\mu)}f(t)dt=\\&\displaystyle\int_{a}^{x}(x-t)^{\beta^*+\mu-1} E_{\alpha,\,\beta^*+\mu}\left( \lambda(x-t)^{\alpha}\right) f(t)dt. 
\end{array}
\end{equation*}
\end{theorem}
\begin{theorem}(see\cite{pod}) Let $\alpha<2,$ $\beta^{*} \in R$ and $\dfrac{\pi\alpha}{2}<\mu < \min\left\lbrace \pi, \pi\alpha\right\rbrace.$ Then we have the following estimate
\begin{equation}\label{MLbound}
\left| E_{\alpha,\beta^{*}}(z)\right| \leq \dfrac{M}{ 1+|z|},\quad \mu \leq |\text{arg} z|\leq \pi,\,|z|\geq 0.\end{equation}
Here and in the rest of the paper, $M$ denotes a positive constant.
\end{theorem}
In the following theorem, we present a homogeneous fractional equation with regularized Caputo-like counterpart hyper-Bessel operator and its explicit solution as proved in \cite{garra}.
\begin{theorem}\label{homsol}( see \cite{garra}, Theorem 2.1)
The function
$$u(t)=E_{\alpha}\left( -\dfrac{\lambda t^{\alpha(1-\theta)}}{(1-\theta)^{\alpha}} \right), $$
solves the fractional Cauchy problem
\[\left\lbrace \begin{array}{rll}
\,^{C}\left(t^{\theta}\dfrac{d}{dt}\right)^{\alpha}u(t)&=-\lambda u(t), &\alpha\in(0,1),\;\theta<1,\; t\geq 0,\\
u(0)&=1.&
\end{array}\right. \]
\end{theorem}
In this paper, we consider a more general case, namely, a non-homogeneous fractional differential equation with a regularized Caputo-like counterpart hyper-Bessel operator presented in the following lemma:
\begin{lemma} \label{nonhomo FDE}
Consider  the following non-homogeneous fractional differential equation
\begin{equation}\label{chp}
\,^{C}\left(t^{\theta}\dfrac{d}{dt}\right)^{\alpha}u(t)=-\lambda u(t)+f(t), \;\alpha\in(0,1),\;\theta<1,\; t\geq 0,
\end{equation}
with
$u(0)=u_0, $ where  $u_0 $ is a constant.
Then, its solution is given in the integral form
\begin{equation}\label{nonsol}
\begin{array}{ll}
u(t)&=u_0\; E_{\alpha}\left( \lambda^{*}t^{\rho\alpha}\right)+
\dfrac{1}{\rho^{\alpha}\Gamma(\alpha)}\displaystyle\int_{0}^{t}(t^{\rho}-x^{\rho})^{\alpha-1}f(x)\, d(x^{\rho})\\
&+\dfrac{\lambda^{*}}{\rho^{\alpha}}\displaystyle\int_{0}^{t}(t^{\rho}-x^{\rho})^{2\alpha-1}E_{\alpha,2\alpha}\left(\lambda^{*} (t^{\rho}-x^{\rho})^{\alpha}\right) f(x)\, d(x^{\rho}),
\end{array}
 \end{equation}
where, $\rho=1-\theta$ and $\lambda^{*}=-\dfrac{\lambda}{\rho^{\alpha}}.$
Moreover, if $f=f_0$ is constant, then the solution reduces to
$$u(t)= C^{*}\;E_{\alpha}\left(\lambda^{*} t^{\rho\alpha}\right)+\dfrac{f_0}{\lambda}.$$
where $C^{*}=\left(u_0-\dfrac{f_0}{\lambda} \right)$.
In particular, when $f=0$, we have
$$u(t)=u_0 \;E_{\alpha}\left(\lambda^{*} t^{\rho\alpha}\right).$$
\end{lemma}
\begin{proof}
First, using relation $(\ref{relation1}),$ equation (\ref{chp}) can be written as
$$ \rho^{\alpha}t^{-\alpha\rho}I_{\rho}^{0,-\alpha} \left( u(t)-u_0 \right)=-\lambda u(t)+f(t),$$
which, on dividing by $\rho^{\alpha}t^{-\alpha\rho},$ becomes
$$I_{\rho}^{0,-\alpha} \left( u(t)-u_0 \right)=\lambda^{*} t^{\rho\alpha}u(t)+ \dfrac{t^{\rho\alpha}}{\rho^{\alpha}}f(t),$$
where $\lambda^{*}=-\dfrac{\lambda}{\rho^{\alpha}}.$ Using the following property of the inverse of  E-K integral (see\cite{mcbride}, Theorem 2.7): $$(I_{m}^{\eta,\alpha})^{-1}=I_{m}^{\eta+\alpha,-\alpha},$$
the above equation can be written as an integral equation, namely,
$$u(t)-\lambda^{*} I_{\rho}^{-\alpha,\alpha}\left( t^{\rho\alpha}u(t)\right)=u_0+\dfrac{1}{\rho^{\alpha}}I_{\rho}^{-\alpha,\alpha}\left( t^{\rho\alpha}f(t)\right),$$
or equivalently,
$$u(t)-\dfrac{\lambda^{*}}{\Gamma(\alpha)}\int_{0}^{t}(t^{\rho}-\tau^{\rho})^{\alpha-1}u(\tau) \, d(\tau^{\rho})=u_0+\dfrac{1}{\rho^{\alpha}\Gamma(\alpha)}\int_{0}^{t}(t^{\rho}-\tau^{\rho})^{\alpha-1}f(\tau)\, d(\tau^{\rho}).$$
Whereupon using Theorem $\ref{eksol}$, we have
\begin{equation}\label{intsol}
\begin{array}{ccc}
 u(t)&=&f^{*}(t)+\lambda^{*}\displaystyle\int_{0}^{t}(t^{\rho}-\tau^{\rho})^{\alpha-1}E_{\alpha,\alpha}(\lambda^{*}(t^{\rho}-\tau^{\rho})^{\alpha})f^{*}(\tau)\, d(\tau^\rho)\\
 &+& u_{0}\left(1+ \lambda^{*}\displaystyle\int_{0}^{t}(t^{\rho}-\tau^{\rho})^{\alpha-1}E_{\alpha,\alpha}(\lambda^{*}(t^{\rho}-\tau^{\rho})^{\alpha})\, d(\tau^\rho)\right),
 \end{array}
\end{equation}
where, $f^{*}(t)=\dfrac{1}{\rho^{\alpha}\Gamma(\alpha)}\displaystyle\int_{0}^{t}(t^{\rho}-x^{\rho})^{\alpha-1} f(x)\, d(x^{\rho}).$\\
The first integral in (\ref{intsol}) can be simplified using Theorem $\ref{intint}$ to the following: 
\begin{equation}\label{intf}
\dfrac{\lambda^{*}}{\rho^{\alpha}}\displaystyle\int_{0}^{t}(t^{\rho}-x^{\rho})^{2\alpha-1}E_{\alpha,2\alpha}\left(\lambda^{*} (t^{\rho}-x^{\rho})^{\alpha}\right) f(x)\, d(x^{\rho}),
\end{equation}
and the second integral in $(\ref{intsol})$ can be also simplified as follows:
\begin{equation}\label{u0}
\begin{array}{l}
u_{0}\left(1+ \lambda^{*}\displaystyle\int_{0}^{t}(t^{\rho}-\tau^{\rho})^{\alpha-1}E_{\alpha,\alpha}(\lambda^{*}(t^{\rho}-\tau^{\rho})^{\alpha})\, d(\tau^\rho)\right)\\
=u_{0}\left( 1+\displaystyle\sum_{k=0}^{\infty} \dfrac{(\lambda^{*})^{k+1}}{\Gamma(\alpha k+\alpha)}\int_{0}^{t}(t^{\rho}-\tau^{\rho})^{\alpha k+\alpha-1} \, d(\tau^{\rho})\right)\\=u_{0}\left( 1+\displaystyle\sum_{k=0}^{\infty} \dfrac{(\lambda^{*})^{k+1}}{\Gamma(\alpha( k+1)+1)}t^{\rho\,\alpha( k+1)}\right)\\
=u_{0}\left( 1+\displaystyle\sum_{m=1}^{\infty} \dfrac{(\lambda^{*})^{m}}{\Gamma(\alpha m+1)}t^{\rho\alpha m}\right)\\
=u_{0}\,E_{\alpha}(\lambda^{*}t^{\rho\alpha}).
\end{array}\end{equation}
Substituting the two simplified forms $(\ref{intf}) $ and $(\ref{u0})$  into $(\ref{intsol}),$  we get the  integral solution $(\ref{nonsol}).$\\
Now, if $ f(t)=f_0$ is constant, then evaluating the first integral in the expression $(\ref{nonsol})$ gives
 $$\dfrac{f_{0}}{\rho^{\alpha}\Gamma(\alpha)}\displaystyle\int_{0}^{t}(t^{\rho}-x^{\rho})^{\alpha-1} \, d(x^{\rho})=\dfrac{f_0 t^{\rho \alpha}}{\rho^{\alpha}\Gamma(\alpha+1)}.$$
Substituting back into $(\ref{nonsol})$ and proceeding in a similar way as in $(\ref{u0})$, the expression of $u(t)$ can be reduced to
$$u(t)=u_{0}E_{\alpha}(\lambda^{*}t^{\rho\alpha})-\dfrac{f_{0} }{\lambda}
\displaystyle\sum_{k=1}^{\infty}\dfrac{\lambda^{*k}\,t^{\rho\alpha k}}{\Gamma(\alpha k+1)},$$
 which can be rewritten as
$$u(t)=\dfrac{f_0}{\lambda}+ C^{*}\; E_{\alpha}\left(\lambda^{*} t^{\rho\alpha}\right), $$
where $C^{*}=\left(u_0-\dfrac{f_0}{\lambda} \right)$. Finally, if $f(t)=f_0=0$, the the expression of $u(x,t)$ can be further reduced to
$$u(t)= u_0\; E_{\alpha}\left(\lambda^{*} t^{\rho\alpha}\right), $$
which is consistent with Theorem \ref{homsol}.
\end{proof}
The rest of this paper is devoted for the main results. In the remaining two sections, we present existence and uniqueness results of solutions to  direct and inverse source problems involving a regularized Caputo-like counterpart hyper-Bessel operator.
\section{A Direct Problem}
\subsection{Statement of Problem and Main Result}
Find a function $u(x,t)$ in a domain $\Omega = \left\{ {0  < x < 1, \, 0 < t < T} \right\}$ satisfying
\begin{equation}\label{direct}
\,^{C}\left(t^{\theta}\dfrac{\partial}{\partial t}\right) ^{\alpha} u(x,t)-u_{xx}(x,t)=f(x,t), \quad (x,t)\in \Omega,
\end{equation}
the boundary conditions
\begin{equation} \label{dbcond}
u(0,t)=0,\quad u(1,t)=0,\quad 0\leq t\leq T,
\end{equation}
and the initial condition
\begin{equation}\label{Icond}
u(x,0)=\psi(x),\quad 0\leq x\leq 1,
\end{equation}
where $f(x,t)$ is a given function, $\theta<1$, $0<\alpha<1$ and $^C\left(t^{\theta}\dfrac{\partial}{\partial t }\right) ^{\alpha}$ is the regularized Caputo-like counterpart hyper-Bessel operator defined in $(\ref{relation})$. Our aim is to prove the existence and uniqueness of solution to the problem (\ref{direct}) - (\ref{Icond}) as stated in the following theorem:
\begin{theorem}
Assume that the following conditions hold
\begin{itemize}
\item $\psi(x)\in C[0,1]$ such that $\psi(0)=\psi(1)=0$ and  $\psi'(x)\in L^{2}(0,1),$
\item $f(\cdot,t)\in C^{3}[0,1]$ such that
$f(0,t)=f(1,t)=f_{xx}(0,t)=f_{xx}(1,t)=0,$ and  $\dfrac{\partial^4}{\partial x^4}
f(\cdot, x)\in L(0,1),$
\end{itemize}
then,  the problem $(\ref{direct})-(\ref{Icond})$ has a unique solution given by
$$
 u(x,t) =\sum_{k=1}^{\infty}  \left[  \psi_k\, E_{\alpha}\left( \dfrac{-k^2 \pi^2}{(1-\theta)^{\alpha}}t^{(1-\theta)\alpha}\right)+F_{k}(t)\right]\sin(k \pi x),$$
where,
$$
\begin{array}{rl}
F_{k}(t)=&\dfrac{1}{(1-\theta)^{\alpha}\Gamma(\alpha)}\displaystyle\int_{0}^{t}\left( t^{(1-\theta)}-\tau^{(1-\theta)}\right) ^{\alpha-1}f_{k}(\tau)\, d(\tau^{(1-\theta)})-\dfrac{k^2\pi^2}{(1-\theta)^{2\alpha}}\\
&\displaystyle\int_{0}^{t}\left( t^{(1-\theta)}-y^{(1-\theta)}\right) ^{2\alpha-1}
E_{\alpha,2\alpha}\left(-\dfrac{\lambda (t^{(1-\theta)}-y^{(1-\theta)})^{\alpha}}{(1-\theta)^{\alpha}}\right)f_{k}(y)\, d(y^{(1-\theta)}),
\end{array}$$
 $$\psi_k = 2 \int_{0}^{1} \psi(x) \sin(k\pi x ) \, dx, \quad k=1,2,3,\cdots$$
$$f_{k}(t)=2 \int_{0}^{1} f(x,t) \sin(k\pi x ) \, dx, \quad k=1,2,3,\cdots$$
\end{theorem}
\subsection{Proof of Result}
\subsubsection{Existence of Solution}
Using separation of variables method for solving the homogeneous equation corresponding to (\ref{direct}) along with the homogeneous boundary conditions  (\ref{dbcond})  yields the following spectral problem:
\begin{eqnarray}\label{sepr}
 \left\lbrace \begin{array}{l}
X''+\lambda X=0,\\
X(0)=0,\quad X(1)=0.
\end{array}\right.
\end{eqnarray}
It is known that the above problem is self adjoint and has the following eigenvalues  $$\lambda_k=(k\pi)^2,\quad k=1,2,3,\cdots$$
and the corresponding eigenfunctions are
\begin{equation} \label{sys1}
X_{k}=\sin(k \pi x) \quad k=1,2,3,\cdots.
\end{equation}
Using the fact that the system of eigenfunctions (\ref{sys1}) forms an orthogonal basis in $L^{2}(0,1)$\cite{Moiseev}, we can write the solution $u(x,t)$ in the form of a series expansion as follows:
\begin{equation}\label{solu}
u(x,t)=\sum_{k=1}^{\infty} u_k(t) \sin(k \pi x), 
\end{equation}
and 
\begin{equation}\label{souf}
f(x,t)=\sum_{k=1}^{\infty} f_k(t) \sin(k \pi x), 
\end{equation}
where, $u_{k}(t)$ is the unknown to be determined and $f_{k}(t)$ is known and given by $$f_{k}(t)=2\int_{0}^{1} f(x,t)\sin(k\pi x)dx.$$
Substituting $(\ref{solu})$ and $(\ref{souf})$ into $(\ref{direct})$ and $(\ref{Icond})$, we get  the linear  fractional differential equation
\begin{equation}\label{fde}
 \,^{C}\left(t^{\theta}\dfrac{d}{dt}\right) ^{\alpha} u_k(t)+ k^2 \pi^2 u_k(t)=f_{k}(t),\end{equation}
with the initial condition
$$u_k(0)=\psi_k,$$
where, $\psi_k$ is the coefficient of the series expansion of $\psi(x)$ in terms of the orthogonal basis (\ref{sys1}), i.e., $$\psi_k =2\int_{0}^{1} \psi(x) \sin(k\pi x ) \, dx.$$
Whereupon using Lemma \ref{nonhomo FDE}, the solution of equation $(\ref{fde})$ is given by
$$
u_{k}(t)=\psi_k \,E_{\alpha}\left( \dfrac{-k^2 \pi^2}{\rho^{\alpha}}t^{\rho\alpha}\right)+ F_{k}(t),
$$ 
where, $\rho=1-\theta$ and 
$$
\begin{array}{rl}
F_{k}(t)=&\dfrac{1}{\rho^{\alpha}\Gamma(\alpha)}\displaystyle\int_{0}^{t}\left( t^{\rho}-\tau^{\rho}\right) ^{\alpha-1}f_{k}(\tau)\, d(\tau^{\rho})\\
-&\dfrac{k^2\pi^2}{\rho^{2\alpha}}\displaystyle\int_{0}^{t}\left( t^{\rho}-y^{\rho}\right) ^{2\alpha-1}E_{\alpha,2\alpha}\left(-\dfrac{\lambda }{\rho^{\alpha}}(t^{\rho}-y^{\rho})^{\alpha}\right)f_{k}(y)\, d(y^{\rho}).
\end{array}$$
Consequently, the expression of $u(x,t)$ can be written as
\begin{equation}
u(x,t) =\sum_{k=1}^{\infty}  \left(  \psi_k\, E_{\alpha}\left( \dfrac{-k^2 \pi^2}{\rho^{\alpha}}t^{\rho\alpha}\right)+F_{k}(t)\right) \sin(k \pi x).
\end{equation}
To complete the proof of existence, we need to prove the uniform convergence of the series representations of $$u(x,t),\,^{C}\left(t^{\theta}\dfrac{\partial}{\partial t }\right) ^{\alpha} u(x,t),\,u_{x}(x,t),\,u_{xx}(x,t).$$
We start with the series representation of $u(x,t)$, rewriting $F_{k}(t)$  as  follows
$$
\begin{array}{rl}
F_{k}(t)=&\dfrac{1}{k^2\pi^2\rho^{\alpha}\Gamma(\alpha)}\displaystyle\int_{0}^{t}\left( t^{\rho}-\tau^{\rho}\right) ^{\alpha-1}f''_{k}(\tau)\, d(\tau^{\rho})\\
-&\dfrac{1}{\rho^{2\alpha}}\displaystyle\int_{0}^{t}\left( t^{\rho}-y^{\rho}\right) ^{2\alpha-1}E_{\alpha,2\alpha}\left(-\dfrac{k^2\pi^2 }{\rho^{\alpha}}(t^{\rho}-y^{\rho})^{\alpha}\right) f''_{k}(y)\, d(y^{\rho}),
\end{array}$$
where,
 $$f''_{k}(t)=2\int_{0}^{1}f_{xx}(x,t)\sin(k\pi x)dx.$$
Now, we estimate the Mittag-Leffler function using inequality $(\ref{MLbound})$ : 
$$\left|  E_{\alpha,2\alpha}\left(-\dfrac{k^2\pi^2 }{\rho^{\alpha}}(t^{\rho}-y^{\rho})^{\alpha}\right)\right|\leq \dfrac{ M\rho^{\alpha}}{\rho^{\alpha}+k^2\pi^2 |t^{\rho}-y^{\rho}|^{\alpha}},$$
which implies the following estimate for $u(x,t)$:
\[\begin{array}{ll}
|u(x,t)|&\leq M\displaystyle \sum_{k=1}^{\infty} \left(  \dfrac{ | \psi_k\,|}{\rho^{\alpha}+k^2\pi^2 |t^{\rho}-y^{\rho}|^{\alpha}}+\dfrac{1}{k^2\pi^2}\displaystyle\int_{0}^{t}\left\vert t^{\rho}-\tau^{\rho}\right\vert ^{\alpha-1}\vert f''_{k}(\tau)\vert d(\tau^{\rho})\right. \\ &\qquad\left. +\displaystyle\int_{0}^{t}\dfrac{ \left\vert t^{\rho}-y^{\rho}\right\vert ^{2\alpha-1}}{\rho^{\alpha}+k^2\pi^2 |t^{\rho}-y^{\rho}|^{\alpha}} \vert f''_{k}(y)\vert\, d(y^{\rho})\right) .
\end{array}\]
Since $\psi(x) \in C[0,1]$ and $f(\cdot, t) \in C^3[0,1]$ , then the above series converges and hence, by Weierstrass M-test the series representation of $u(x,t)$ is  uniformly convergent in $\Omega$.\\
Next, we show the uniform convergence of  series representation of $u_{xx}(x,t)$, which is given by
$$
u_{xx}(x,t) =-\sum_{k=1}^{\infty}k^2\pi^2  \left(  \psi_k\, E_{\alpha}\left( \dfrac{-k^2 \pi^2}{\rho^{\alpha}}t^{\rho\alpha}\right)+F_{k}(t)\right) \sin(k \pi x).$$
To prove this assertion, we have the following estimate
\[\begin{array}{ll}
|u_{xx}(x,t)|&\leq M\displaystyle \sum_{k=1}^{\infty} \left(  \dfrac{ k^2\pi^2| \psi_k\,|}{\rho^{\alpha}+k^2\pi^2 (t^{\rho}-y^{\rho})^{\alpha}}+\dfrac{1}{k^2\pi^2}\displaystyle\int_{0}^{t}\left\vert t^{\rho}-\tau^{\rho}\right\vert ^{\alpha-1}\vert f^{(4)}_{k}(\tau)\vert d(\tau^{\rho})\right. \\ &\qquad\left. +\displaystyle\int_{0}^{t}\dfrac{ \left\vert t^{\rho}-y^{\rho}\right\vert ^{2\alpha-1}}{\rho^{\alpha}+k^2\pi^2 (t^{\rho}-y^{\rho})^{\alpha}} \vert f^{(4)}_{k}(y)\vert\, d(y^{\rho})\right),\end{array}\]
where
$$f_{k}^{(4)}(t)=2\int_{0}^{1}\dfrac{\partial^4}{dx^{4}}f(x,t) \sin(k\pi x) \,dx.$$
Since $\psi(0)=\psi(1)=0$ and $ \dfrac{\partial^4 f}{\partial x^4}(\cdot,t) \in L(0,1)$, then using integration by parts, we arrive at the following estimate
\[\begin{array}{ll}
|u_{xx}(x,t)|&\leq M \displaystyle\sum_{k=1}^{\infty} \left(  \dfrac{1}{k\pi}\left|\psi_{k}^{(1)}\right|+\dfrac{1}{k^2\pi^2} \right)\\ 
&\leq M\left(  \displaystyle\sum_{k=1}^{\infty} \dfrac{1}{(k\pi)^2}+\displaystyle\sum_{k=1}^{\infty} \left|\psi_{k}^{(1)} \right|^2 \right), 
\end{array}\]
where, we have used the inequality $2ab \le a^2 + b^2$ and 
$$
\psi_{k}^{(1)}=2\displaystyle\int_{0}^{1}\psi'(x) \cos(k\pi x) \,dx.$$
Then, Bessel's inequality for trigonometric functions $$ \sum_{k=0}^{\infty} g_{k}^2\leq \Vert g\Vert_{L^2(0,1)}^2,$$ implies
\[\begin{array}{ll}
|u_{xx}(x,t)|&\leq M  
 \displaystyle\sum_{k=1}^{\infty} \dfrac{1}{(k\pi)^2}+\Vert\psi'(x)\Vert_{L^2(0,1)}^2. 
\end{array}\]
Thus, the series in expression of $ u_{xx}(x,t)$ is bounded by a convergent series which implies that its uniformly convergent by Weierstrass M-test.
Finally, series representation of 
 $\,^{C}\left(t^{\theta}\dfrac{\partial}{\partial t}\right) ^{\alpha}u(x,t)$ is given by
$$
\begin{array}{ll}
\,^{C}\left(t^{\theta}\dfrac{\partial}{\partial t}\right) ^{\alpha}u(x,t)&=-\displaystyle\sum_{k=1}^{\infty}k^2 \pi^2\left(  \psi_k  E_{\alpha}\left( \dfrac{-k^2 \pi^2 }{\rho^{\alpha}}+ t^{\rho\alpha}\right)+F_{k}(t)\right)   \sin(k \pi x)
+\,f(x,t),
\end{array}
$$
and convergence of the above series follows directly from the uniform convergence of  $u_{xx}(x,t),$
which also ensures the uniform convergence of $u_{x}(x,t).$ 
%the series representation of $\,^{C}\left(t^{\theta}\dfrac{\partial}{\partial t}\right) ^{\alpha}u(x,t).$ 
%which, in turn, implies the uniform convergence of $u_{x}(x,t). $
\subsubsection{Uniqueness of Solution:}
Suppose that  $u_1(x,t)$ and  $u_2(x,t)$ are two solutions of the problem (\ref{direct}) - (\ref{Icond}), then $\widehat{u}(x,t)= u_1(x,t)-  u_2(x,t)$  satisfies the following boundary value problem:
\begin{eqnarray}\label{eq} 
&\,^{C}\left(t^{\theta}\dfrac{\partial}{\partial t }\right) ^{\alpha} \widehat{u}-\dfrac{\partial^2 \widehat{u}}{\partial x^2}=0,&(x,t)\in \Omega,
\\
\label{condb}
&\widehat{u}(0,t)=0,\quad \widehat{u}(1,t)=0, &0\leq t \leq T,\\
\label{condi}
&\widehat{u}(x,0)=0, &0\leq x \leq 1.
\end{eqnarray}
Define the following  function:
\begin{equation}\label{un}
u_{k}(t)=2\int_{0}^{1} \widehat{u}(x,t) \sin(k\pi x)dx.\end{equation}
Then, the  initial condition $(\ref{condi})$  implies  
\begin{equation}\label{newicond}
u_{k}(0)=0.\end{equation}
 Applying regularized Caputo-like counterpart hyper-Bessel operator to $(\ref{un})$, we get
$$
\begin{array}{ll}
\,^{C}\left(t^{\theta}\dfrac{d}{d t }\right) ^{\alpha} u_{k}(t)&=2\displaystyle\int_{0}^{1}\,^{C}\left(t^{\theta}\dfrac{\partial}{\partial t }\right) ^{\alpha} \widehat{u}(x,t)\sin(k\pi x)dx,\\
&=2\displaystyle\int_{0}^{1}  \widehat{u}_{xx}(x,t)\sin(k\pi x)dx.\\
\end{array} $$ 
Then, integrating by parts twice and using boundary conditions $(\ref{condb})$, we obtain the following fractional differential equation
$$\,^{C}\left(t^{\theta}\dfrac{d}{d t }\right) ^{\alpha}u_{k}(t)-(k\pi)^2 u_{k}(t)=0.$$
Using Lemma $\ref{nonhomo FDE}$, the above equation with the initial condition $(\ref{newicond})$ has the trivial solution  $u_{k}(t)\equiv 0,$
and hence we have  $$\int_{0}^{1}\widehat{u}(x,t)\sin(k\pi x)dx=0.$$
Therefore,  using the completeness property of system $(\ref{sys1})$, we deduce that
$\widehat{u}(x,t)=0$ in $\Omega$, which implies the uniqueness of solution to the problem (\ref{direct}) - (\ref{Icond}).
 
\section{Inverse source problem}
 Here, we consider an inverse source problem of finding a pair of functions $\left\lbrace u(x,t), f(x) \right\rbrace $ in a  rectangular domain $\Omega=\left\lbrace (x,t):0<x<1,\; 0<t<T\right\rbrace,$ which satisfies the following  initial-boundary value problem:
\begin{eqnarray}
&&\,^{C}\left(t^{\theta}\dfrac{\partial}{\partial t}\right) ^{\alpha} u(x,t)-u_{xx}(x,t)=f(x), \quad \quad (x,t) \in \Omega \label{sinverse}\\
&&u(0,t)=0,\quad u(1,t)=0,\hspace{3cm} 0\leq t\leq T,\label{ISBC}\\
&&u(x,0)=\psi(x),\quad u(x,T)=\phi(x), \hspace{1.5cm} 0\leq x\leq 1, \label{ISIC}
\end{eqnarray}
where $\phi$  and $\psi$ are given functions, such that
\[
\psi(0)=\psi(1)=0, \quad \phi(0)=\phi(1)=0,
\]
which follows directly from (\ref{ISBC}) and (\ref{ISIC}). As in the previous section, we seek solution to problem (\ref{sinverse}) - (\ref{ISIC}) in a form of series expansions using the orthogonal system (\ref{sys1}) as follows:
$$
u(x,t)=\sum_{k=1}^{\infty} u_k(t) \sin(k \pi x),
$$
$$
f(x)=\sum_{k=1}^{\infty} f_k \sin(k \pi x).
$$
where $f_k$, $u_k$ are the unknowns to be determined.
Substituting the above expressions for $u(x,t)$ and $f(x)$ into (\ref{sinverse}) and (\ref{ISIC}) gives the following  fractional differential equation:
$$ \,^{C}\left(t^{\theta}\dfrac{d}{dt}\right) ^{\alpha}u_k(t)+k^2 \pi^2    u_k(t)=f_k,$$
with the following conditions
$$u_k(0)=\psi_k,\quad u_{k}(T)=\phi_k,$$
where $\psi_k, \phi_k$ are called the Fourier sine coefficients and defined as $$\psi _k=2\int_{0}^{1}\psi (x)\sin(k\pi x) dx, \quad \phi_k=2\int_{0}^{1}\phi(x)\sin(k\pi x) dx.$$
Solving the above equation, using Lemma $2.7$, we obtain
$$u_k(t)= C_k \;E_{\alpha}\left(-\dfrac{k^2\pi^2}{(1-\theta)^{\alpha}} t^{(1-\theta)\alpha}\right)+\dfrac{f_k}{k^2\pi^2},$$
and using the given initial conditions,  we have
$$C_k=\dfrac{\psi_k-\phi_k}{1-E_{\alpha}\left(-\dfrac{k^2\pi^2}{(1-\theta)^{\alpha}} T^{(1-\theta)\alpha}\right)},\qquad f_k=k^2\pi^2(\psi_k - C_k).$$
Hence, the expressions for $u(x,t)$ and $f(x)$ can be written as,
\[\begin{array}{ll}
u(x,t)&=\displaystyle\sum_{k=1}^{\infty} C_k E_{\alpha}\left(-\dfrac{k^2\pi^2}{(1-\theta)^{\alpha}} t^{(1-\theta)\alpha}\right)\sin(k \pi x)+ (\psi_k - C_k)\sin(k \pi x) \\
&=\psi(x)-\displaystyle\sum_{k=1}^{\infty}\dfrac{1-E_{\alpha}\left(-\dfrac{k^2\pi^2}{(1-\theta)^{\alpha}} t^{(1-\theta)\alpha}\right)}{1-E_{\alpha}\left(-\dfrac{k^2\pi^2}{(1-\theta)^{\alpha}} T^{(1-\theta)\alpha}\right)}(\psi_k-\phi_k)\sin(k \pi x),\end{array}\]
and
\[\begin{array}{ll}
f(x)&=\displaystyle \sum_{k=1}^{\infty}k^2\pi^2 (\psi_k-C_k)\sin(k \pi x)\\
&=\psi''(x)-\displaystyle\sum_{k=1}^{\infty}\dfrac{k^2\pi^2 (\psi_k-\phi_k)}{1-E_{\alpha}\left(-\dfrac{k^2\pi^2}{(1-\theta)^{\alpha}} T^{(1-\theta)\alpha}\right)}\sin(k \pi x).\end{array}\]
Appropriate conditions on the given functions $\psi(x)$ and $\phi(x)$, see the Theorem $4.1$ below, are assumed for establishing the uniform convergence of the series expansions of  $u(x,t),$ $\,^{C}\left(t^{\theta}\dfrac{\partial}{\partial t}\right)^{\alpha}u(x,t),$ $\,u_{x}(x,t)$, $u_{xx}(x,t)$ and  $f(x).$ This can be done in a similar approach as presented earlier. For example,  for $f(x)$ we have the following estimate: 
\[\begin{array}{ll}
|f(x)|&\leq |\psi''(x)|+ \displaystyle\sum_{k=1}^{\infty}\dfrac{ k^2\pi^2 \left[ (1-\theta)^{\alpha}+k^2\pi^2 T^{ (1-\theta)\alpha}\right] }{(1-M)(1-\theta)^{\alpha}+k^2\pi^2 T^{ (1-\theta)\alpha}}\left( \left|\psi_{k}\right|+\left|\phi_{k}\right|\right)  \\
&\leq |\psi''(x)|+M \displaystyle\sum_{k=1}^{\infty}\dfrac{ 1}{k\pi}\left( \left|\psi_{k}^{(3)}\right|+\left|\phi_{k}^{(3)}\right|\right)  \\
&\leq |\psi''(x)|+ M \left(\displaystyle \sum_{k=1}^{\infty}\dfrac{1}{(k\pi)^2}+\displaystyle \sum_{k=1}^{\infty}\left|\psi_{k}^{(3)}\right|^2+ \sum_{k=1}^{\infty}\left|\phi_{k}^{(3)}\right|^2\right) \\
&\leq |\psi''(x)|+M\left(\displaystyle \sum_{k=1}^{\infty}\dfrac{1}{(k\pi)^2}+\Vert\psi'''(x)\Vert_{L^2(0,1)}^2+\Vert\phi'''(x)\Vert_{L^2(0,1)}^2\right),
\end{array}\]
where
$$
\psi_{k}^{(3)}=2\displaystyle\int_{0}^{1}\psi'''(x) \cos(k\pi x) \,dx,$$
and $$
\phi_{k}^{(3)}=2\displaystyle\int_{0}^{1}\phi'''(x) \cos(k\pi x) \,dx.$$
 Assuming that $\psi(x)\in C^{2}[0,1]$ and $ \psi'''(x),\phi'''(x)\in L^{2}(0,1),$ then by Weierstrass M-test the series representation of $f(x)$ is uniformly convergent.
Also, the series representation of $\,^{C}\left(t^{\theta}\dfrac{\partial}{\partial t}\right)^{\alpha}u(x,t)$, which is given by 
$$
\,^{C}\left(t^{\theta}\dfrac{\partial}{\partial t}\right)^{\alpha}u(x,t)
=-\displaystyle\sum_{k=1}^{\infty}\dfrac{E_{\alpha}\left(-\dfrac{k^2\pi^2}{(1-\theta)^{\alpha}} t^{(1-\theta)\alpha}\right)}{1-E_{\alpha}\left(-\dfrac{k^2\pi^2}{(1-\theta)^{\alpha}} T^{(1-\theta)\alpha}\right)}k^2\pi^2(\psi_k-\phi_k)\sin(k \pi x),
 $$
can be estimated as follows:
\[\begin{array}{ll}
\left|\,^{C}\left(t^{\theta}\dfrac{\partial}{\partial t}\right)^{\alpha}u(x,t)\right|
&\leq M\displaystyle\sum_{k=1}^{\infty}\dfrac{ k^2\pi^2}{ (1-\theta)^{\alpha}+k^2\pi^2 t^{ (1-\theta)\alpha}  }(\left|\psi_k\right|+\left|\phi_k\right|)\\
&\leq M\displaystyle\sum_{k=1}^{\infty}\dfrac{1}{k\pi}\left( \left|\psi_k^{(1)}\right|+\left|\phi_k^{(1)}\right|\right)\\
&\leq M\left(\displaystyle \sum_{k=1}^{\infty}\dfrac{1}{(k\pi)^2}+\Vert\psi'(x)\Vert_{L^2(0,1)}^2+\Vert\phi'(x)\Vert_{L^2(0,1)}^2\right).
\end{array}\]
It is clear that the above series is uniformly convergent.
The main result for this section can be summarized in the following theorem:
\begin{theorem}
Assume $\psi(x),\phi(x)\in C^{2}[0,1]$ such that
 $\psi^{(i)}(0)=\psi^{(i)}(1)=\phi^{(i)}(0)=\phi^{(i)}(1)=0,$ $(i=0,2)$  and $\psi'''(x),\phi'''(x)\in L^{2}(0,1),$
then the inverse source problem (\ref{sinverse}) - (\ref{ISIC}) has  a unique pair of solutions $\left\lbrace u(x,t), f(x) \right\rbrace $  given by
$$
u(x,t)=\psi(x)-\displaystyle\sum_{k=1}^{\infty}\dfrac{1-E_{\alpha}\left(-\dfrac{k^2 \pi^2}{(1-\theta)^{\alpha}} t^{(1-\theta)\alpha}\right)}{1-E_{\alpha}\left(-\dfrac{k^2 \pi^2}{(1-\theta)^{\alpha}} T^{(1-\theta)\alpha}\right)}(\psi_k-\phi_k)\sin(k \pi x),
$$
$$
f(x)=\psi''(x)-\displaystyle\sum_{k=1}^{\infty}\dfrac{k^2 \pi^2 (\psi_k-\phi_k)}{1-E_{\alpha}\left(-\dfrac{k^2 \pi^2}{(1-\theta)^{\alpha}} T^{(1-\theta)\alpha}\right)}\sin(k \pi x),$$
where,
$$\psi _k=2\int_{0}^{1}\psi (x)\sin(k\pi x) dx, \quad \phi_k=2\int_{0}^{1}\phi(x)\sin(k\pi x) dx.$$
\end{theorem}
\begin{description}
\item[Acknowledgements.]
Authors acknowledge financial support from The Research Council (TRC), Oman. This work is funded by TRC under the research agreement no. ORG/SQU/CBS/13/030.
\end{description}
%%%%%%%%%%%%%%%%%%%%%%%%%%%%%%%%%%%%%%%%%%%%%%%%
%%%%%%%%%%%%%%%%%%%%%%%%%%%%%%%%%%%%%%%%%%%%%%%%
%%%%%%%%%%%%%%%%%%%%%%%%%%%%%%%%%%%%%%%%%%%%%%%%
%%%%%%%%%%%%%%%%%%%%%%%%%%%%%%%%%%%%%%%%%%%%%%%%


\begin{thebibliography}{99}
\bibitem{kiry}B. AL-Saqabi and V. Kiryakova, Explicit solutions of fractional integral and differential equations involving Erd\'elyi-Kober operators, Applied Mathematics and Computations, 95 (1998), 1--13.

\bibitem{dim}I. Dimovski, Operational calculus of a class of differential operators, C.R. Acad. Bulg. Sci., 19 (12) (1966), 1111--1114.

%\bibitem{garra0}M. Concezzi, R. Garra and R. Spigler, Fractional Relaxation and fractional oscillation models involving Erd\'elyi-Kober integral, Fractional Calculus and Applied Analysis, Volume 18, pp. 1212 –1231 ,(2015).

\bibitem{garra}R. Garra, A. Giusti, F. Mainardi  and G. Pagnini, Fractional relaxation with time-varying coefficient,  Fractional Calculus and Applied Analysis, 17 (2014), 424--439.

\bibitem{garra1}R. Garra, E. Orsingher and F. Polito, Fractional diffusion with time-varying coefficients, Journal of Mathematical Physics, AIP Conf. Proc., 56 (2015) 1--19.

\bibitem{gm} R. Gorenflo, F. Mainardi,  Fractional calculus: integral and differential equations of fractional order, Fractals and Fractional Calculus in continuum Mechanics : A. Carpinteri and F. Mainardi (eds), Springer Verlag, Wien and New York , (1997), 223--276.
%\bibitem{kiry1}V. Kiryakova, Fractional order differential and integral equations with  Erd\'elyi-Kober operators : explicit solutions by means of the transformation method, Applications of Mathematics in Engineering and Economics, AIP Conf. Proc.1410, pp. 247-258, (2011).

\bibitem{kiry2}V. Kiryakova, Generalized Fractional Calculus and Applications. Longman- J. Wiley, Harlow- N.York, 1994.

\bibitem{fbm}B. B. Mandelbrot and J.W. Van Ness, Fractional Brownian motions, Fractional noises and applications, SIAM Review, 10 (4) (1968), 422--437.

\bibitem{mcbride} A.C. McBride, Fractional Powers of a Class of Ordinary Differential Operators. Proceedings of the London Mathematical Society, 3 (45) (1982), 519--546.

\bibitem{Moiseev} E. I. Moiseev, On the basis property of systems of sines and cosines, Doklady AN SSSR, 275 (4) (1984) 794--798.

\bibitem{pod} I. Podlubny,  Fractional differential equations. Academic Press Inc., San Diego, CA, 1999.
\bibitem{Prabhakar} T. R. Prabhakar, A singular integral equation with a generalized Mittag-Leffler function in the kernal, Yokohama Math. J., 19 (1971), pp. 7-15.



\end{thebibliography}
\end{document}